\documentclass[12pt,leqno]{article}
\usepackage{amsmath,amssymb,amsthm,amscd,dsfont}
\numberwithin{equation}{section} \allowdisplaybreaks
\begin{document}
\newtheorem{theorem}{Theorem}[section]
\newtheorem{defin}{Definition}[section]
\newtheorem{prop}{Proposition}[section]
\newtheorem{corol}{Corollary}[section]
\newtheorem{lemma}{Lemma}[section]
\newtheorem{rem}{Remark}[section]
\newtheorem{example}{Example}[section]
\title{A note on submanifolds of generalized K\"ahler manifolds}
\author{{\small by}\vspace{2mm}\\Izu Vaisman}
\date{}
\maketitle
{\def\thefootnote{*}\footnotetext[1]%
{{\it 2010 Mathematics Subject Classification: 53C15} .
}}
\begin{center} \begin{minipage}{12cm}
A{\footnotesize BSTRACT. In this note, we consider submanifolds of a generalized K\"ahler manifold that are CR-submanifolds for the two associated Hermitian structures. Then, we establish the conditions for the induced, generalized F structure to be a CRFK structure. The results extend similar conditions which we obtained for hypersurfaces in an earlier paper.}
\end{minipage}
\end{center} \vspace{5mm}
\section{Introduction}
This note is a complement to our previous paper \cite{Vnou}. All manifolds and mappings are of class $C^\infty$ and the terminology and notation are classical \cite{KN}. An exception is the use of Cartan's conventions for exterior products and differentials, e.g.,
$$\begin{array}{l}
\alpha\wedge\beta(X,Y)=\alpha(X)\beta(Y)-\alpha(Y)\beta(X), \vspace*{2mm}\\ d\alpha(X,Y)=X\alpha(Y)-Y\alpha(X)-\alpha([X,Y]).\end{array}$$
Furthermore, we shall assume that the reader is familiar with the basic notions and facts of generalized geometry in the sense of Hitchin as they already appeared in many papers. In particular, we shall refer to \cite{{Gualt},{V},{V1},{V3}}.

In this note we consider a class
of submanifolds of a generalized K\"ahler manifold, which bear a naturally induced generalized metric F structure and we study the conditions for the induced structure to be a CRFK structure\footnote{CR stands for Cauchy-Riemann, F stands for Yano's F structure and K comes from K\"ahler.} in the sense of \cite{V1}. In \cite{Vnou} we studied this problem in the case of hypersurfaces.

First, we shall deduce a result in the classical framework. Namely, we consider a CR-submanifold of a Hermitian manifold and its induced F structure \cite{Bej} and we establish the conditions for the latter to be classical CRF in the sense of \cite{V1}. As a corollary, it follows that these conditions hold for totally geodesic and totally umbilical CR-submanifolds. Then, we shall consider generalized CR-submanifolds of a generalized K\"ahler manifold, i.e. submanifolds that have the CR property for the two associated Hermitian structures. Generalized CR-submanifolds have an induced, generalized F structure and we establish the conditions for the induced structure to be CRFK. As a corollary, it follows that, if the generalized CR-submanifold is totally geodesic, the induced structure is a generalized CRFK structure.
\section{Generalized CR-submanifolds}
Let $M^{2n}$ be a generalized almost Hermitian manifold, with the generalized Riemannian metric $G$ and the compatible generalized almost complex structure $\mathcal{J}$. Then, the following results hold \cite{Gualt}.

$G$ is equivalent to $\mathcal{G}\in End(\mathbf{T}M)$ ($\mathbf{T}M=TM\oplus T^*M$) defined by
$$ G(\mathcal{G}\mathcal{X},\mathcal{Y})=g(\mathcal{X},\mathcal{Y})=
\frac{1}{2}(\alpha(Y)+\beta(X)),$$
where $\mathcal{X}=(X,\alpha),\mathcal{Y}=(Y,\beta)\in\mathbf{T}M$ and
$$\mathcal{G}^2=Id,\,g(\mathcal{G}\mathcal{X},\mathcal{G}\mathcal{Y})= g(\mathcal{X},\mathcal{Y}).$$ \indent
$G$ is also equivalent to a pair $(\gamma,\psi)$ where $\gamma$ is a Riemannian metric and $\psi$ is a $2$-form on $M$. The equivalence is via the $\pm1$-eigenbundles of $\mathcal{G}$
$$ V_\pm=\{(X,\flat_{\psi\pm\gamma}X), X\in TM\}\; (\flat_{\psi\pm\gamma}X=i(X(\psi\pm\gamma))$$ and the projections $\tau_\pm=pr_{TM}: V_\pm\rightarrow TM$ are {\it transfer isomorphisms}.

For the structure $\mathcal{J}$ one has
$$\mathcal{J}^2=-Id,\;g(\mathcal{J}\mathcal{X},\mathcal{Y}) +g(\mathcal{X},\mathcal{J}\mathcal{Y})=0,\; G(\mathcal{J}\mathcal{X},\mathcal{J}\mathcal{Y})= G(\mathcal{X},\mathcal{Y}).$$
The bundles $V_\pm$ are $\mathcal{J}$-invariant and the transfer by $\tau_\pm$ produces two $\gamma$-compatible almost complex structures $J_\pm$ of $M$ such that
$$ \mathcal{J}(X,\flat_{\psi\pm\gamma}X)=(J_\pm X,\flat_{\psi\pm\gamma}(J_\pm X)).$$
Thus, $(G,\mathcal{J})$ is equivalent to the quadruple $(\gamma,\psi,J_\pm)$.

Furthermore, a complementary, $G$-compatible, generalized almost complex structure is defined by $\mathcal{J}'=\mathcal{G}\circ\mathcal{J}=\mathcal{J}\circ\mathcal{G}$ and $\mathcal{J}\circ\mathcal{J}'=\mathcal{J}'\circ\mathcal{J},\,\mathcal{G} =-\mathcal{J}\circ\mathcal{J}'$. The complementary structure corresponds to $(\gamma,\psi,J_+,-J_-)$.

On an arbitrary manifold $M$, a {\it generalized F structure} $\mathcal{F}\in End\,\mathbf{T}M$ \cite{V1} is defined by the conditions
$$\mathcal{F}^3+\mathcal{F}=0,\; g(\mathcal{F}\mathcal{X},\mathcal{Y}) +g(\mathcal{X},\mathcal{F}\mathcal{Y})=0$$
and the structure is {\it metric} with respect to a generalized Riemannian metric $G$ if
$$ G(\mathcal{F}\mathcal{X},\mathcal{Y})+G(\mathcal{X},\mathcal{F}\mathcal{Y})=0.
$$
Then, like in the almost Hermitian case, there exists a complementary generalized metric F structure $\mathcal{F}'=\mathcal{G}\circ\mathcal{F}$.

By Proposition 4.2 of \cite{V1}, $(\mathcal{F},G)$ is a generalized metric F structure iff there exists two classical metric F structures $(F_\pm,\gamma)$ on $M$, i.e.,
\begin{equation}\label{Fclasmetric}	 F_\pm^3+F_\pm=0,\;
\gamma(F_\pm X,Y)+\gamma(X,F_\pm Y)=0\;(X,Y\in TM)\end{equation}
and the generalized F structure is given by $$\mathcal{F}(X,\flat_{\psi \pm\gamma}X)=(F_\pm X,\flat_{\psi\pm\gamma}(F_\pm X)),\;\forall X\in TM.$$

Now, let $\iota:N^k\hookrightarrow M$ be a submanifold of $M$ and let $\nu N=T^{\perp_\gamma}N$ be the normal bundle of $N$. Then, $T_NM=TN\oplus\nu N$ and we shall identify $T^*N=ann\,\nu N,\nu^* N=ann\,TN,\mathbf{T}N=TN\oplus ann\,\nu N$. It follows easily that
$$\mathbf{T}^{\perp_g}N=\nu N\oplus ann\,TN,$$
hence, the restriction $g|_{\mathbf{T}N}$ coincides with the pairing metric on the manifold $N$, thus, it is non degenerate, and
\begin{equation}\label{gGdesc} \mathbf{T}_NM=\mathbf{T}N\oplus\mathbf{T}^{\perp_g}N. \end{equation}

The metric $G$ induces a generalized Riemannian metric $G'$ on $N$ that corresponds to the induced pair $(\gamma'=\iota^*\gamma,\psi'=\iota^*\psi)$ and has the $\pm1$-eigenbundles $$V'_\pm=\{(X,pr_{ann\,\nu N}(\flat_{\psi\pm\gamma}X))\,/\,X\in TN\}= pr_{\mathbf{T}N}V_\pm,$$ where the projection is defined by (\ref{gGdesc}) (e.g., \cite{V4}). In the particular case $\psi=0$, we get $V'_\pm=V_\pm\cap\mathbf{T}N$, we have
\begin{equation}\label{psi01} \mathbf{T}N=(V_+\cap\mathbf{T}N)\oplus(V_-\cap\mathbf{T}N) \subseteq \mathbf{T}M\end{equation}
and $G'$ is induced by $G$ via the inclusion (\ref{psi01}).

Now, we define the class of submanifolds that we want to study.
\begin{defin}\label{defmfsb} {\rm 1. If $(M,\gamma,J)$ is a classical almost Hermitian manifold, a submanifold $\iota:N\hookrightarrow M$ is called a {\it CR-submanifold} if the equality
\begin{equation}\label{restrcond} TN=(TN\cap J(TN))\oplus (TN\cap J(\nu N))\end{equation}
holds at every point of $N$ and the rank of the terms is constant.

2. If $(M,\gamma,\psi,J_\pm)$ is a generalized almost Hermitian manifold, a submanifold $\iota:N\hookrightarrow M$ is called a {\it generalized CR-submanifold} if it is a CR-submanifold with respect to the two almost Hermitian structures $(\gamma,J_\pm)$.}\end{defin}

Part 1 of Definition \ref{defmfsb} is equivalent to Bejancu's original definition \cite{Bej}, the distributions $D,D^\perp$ of \cite{Bej} being the terms of the direct sum (\ref{restrcond}). Among the examples of CR-submanifolds we notice the hypersurfaces and the $\Omega$-coisotropic submanifolds ($\Omega(X,Y)=\gamma(JX,Y)$ is the K\"ahler form). The CR terminology is justified by the fact that, if $J$ is integrable with $i$-eigenbundle $L\subseteq T^cM$, then, $L\cap T^cN$ is a CR structure (the index $c$ denotes complexification). In the particular case $\psi=0$, if we apply the transfer $\tau_\pm^{-1}$ to the equalities (\ref{restrcond}) for $J_\pm$ and use (\ref{psi01}), we can see that the generalized CR-submanifolds are characterized by
$$\mathbf{T}N=(\mathbf{T}N\cap\mathcal{J}(\mathbf{T}N)) \oplus(\mathbf{T}N\cap\mathcal{J}(\nu N\oplus\nu^*N)),$$
hence, if $\psi=0$, a generalized CR-submanifold is an F submanifold in the sense of \cite{V1}.

If (\ref{restrcond}) holds, $N$ has the induced metric F structure
\begin{equation}\label{clasind} F|_{TN\cap J_\pm(TN)}= J|_{TN\cap J_\pm(TN)},\;F|_{TN\cap J(\nu N)}=0. \end{equation}
Notice that $F$ of (\ref{clasind}) coincides with the tensor $\phi$ of \cite{Bej}.

In the generalized case, the use of $J_\pm$ in (\ref{clasind}) yields two structures $F_\pm$ and we get an {\it induced} generalized metric F structure $\mathcal{F}$ defined by the quadruple $(\gamma,\psi,F_\pm)$.
\begin{center}
\section{Submanifolds of generalized K\"ahler manifolds}\end{center}
The generalized almost complex structure $\mathcal{J}$ may be identified with its $\pm i$-eigenbundles $\mathcal{L},\bar{\mathcal{L}}$ (the bar denotes complex conjugation) and in the generalized almost Hermitian case of $(G,\mathcal{J},\mathcal{J}')$ one has \cite{Gualt}
$$\mathcal{L}=(\mathcal{L}\cap V_+)\oplus(\mathcal{L}\cap V_-),\,\mathcal{L}'=(\mathcal{L}\cap V_+)\oplus(\bar{\mathcal{L}}\cap V_+).$$

The structure $\mathcal{J}$ is {\it integrable} ({\it generalized complex}), respectively $(G,\mathcal{J},\mathcal{J}')$ is {\it generalized Hermitian}, if $\mathcal{L}$ is closed under Courant brackets. Furthermore, $(G,\mathcal{J},\mathcal{J}')$ is {\it generalized K\"ahler} if $\mathcal{J},\mathcal{J}'$ are integrable, which turns out to be equivalent to the following pair of properties: (i) the pairs $(\gamma,J_\pm)$ are Hermitian structures, (ii) for the Hermitian structures $(\gamma,J_\pm)$ one has
\begin{equation}\label{relpsiJ} \gamma(\nabla^\gamma_XJ_\pm(Y),Z)=\mp\frac{1}{2}[d\psi(X,J_\pm Y,Z)+d\psi(X,Y,J_\pm Z)],
\end{equation}
where $\nabla^\gamma$ is the Levi-Civita connection of $\gamma$ \cite{{Gualt},{V3}}.
If the $2$-form $\psi$ is closed, $M$ is a {\it bi-K\"ahlerian manifold}, i.e., a manifold with two K\"ahler structures with the same Riemannian metric $\gamma$.

The generalized F structure $\mathcal{F}$ may be identified with its $(\pm i,0)$-eigenbundles $\mathcal{E},\bar{\mathcal{E}},\mathcal{S}$ and $\mathcal{F}$ is {\it integrable {\rm or} CRF} if $\mathcal{E}$ is closed under Courant brackets \cite{V1}. Furthermore, the generalized metric F structure $(G,\mathcal{F},\mathcal{F}')$ is a {\it generalized CRFK structure} if $\mathcal{F},\mathcal{F}'$ are integrable and the eigenbundles of $\mathcal{G},\mathcal{F}$ satisfy the Courant bracket condition
$$[V_+\cap\mathcal{S},V_-\cap\mathcal{S}]\subseteq\mathcal{S}.$$ These properties are equivalent to the pair of properties \cite{V1} (a) the corresponding structures $F_\pm$ are classical CRF structures, i.e., $\mathcal{F}_\pm(X,\alpha)=(F_\pm X,-\alpha\circ F_\pm)$ are generalized CRF structures, (b) one has the equalities
\begin{equation}\label{CRFK6} \gamma(F_\pm(\nabla^\gamma_XF_\pm)Y,Z)= \pm\frac{1}{2}[d\psi(X,Y,F^2_\pm Z)+d\psi(X,F_\pm Y,F_\pm Z)].\end{equation}
If the form $\psi$ is closed, $M$ is a {\it partially bi-K\"ahlerian submanifolds}, i.e., a Riemannian manifold such that its metric $\gamma$ has two de Rham decompositions that have one K\"ahlerian term \cite{V1}.

Hereafter, we shall assume that $(M,G,\mathcal{J},\mathcal{J}')$ is a generalized K\'ahler manifold, $N$ is a generalized CR-submanifold and $(G',\mathcal{F})$ is the induced structure. Then, we will look for the conditions that characterize the case where the induced structure is a CRFK structure
and we begin by the following preparations.

Riemannian geometry gives us the Gauss-Weingarten equations along the submanifold $N$ of the Riemannian manifold $(M,\gamma)$,
\begin{equation}\label{G-W} \nabla^\gamma_XY=\nabla^{\gamma'}_XY+b(X,Y),\;\nabla^\gamma_XU=-W_U X+\nabla^\nu_XU,\end{equation} where $X,Y\in TN,U\in\nu N$, $\nabla^\gamma,\nabla^{\gamma'}$ are the Levi-Civita connections of the metrics $\gamma,\gamma'=\iota^*\gamma$,
$\nabla^\nu$ is the induced connection of the normal bundle of $N$ and $b(X,Y)=b(Y,X)\in\nu N,W_\nu X\in TN$ are the $\nu N$-valued second fundamental form and the Weingarten operator, respectively. The latter are related by the formula $\gamma(W_U X,Y)=\gamma(b(X,Y),U)$.

Using these equations, we can extend the proof of Proposition 3.2 of \cite{Vnou} and get the following result.
\begin{theorem}\label{CRFclind} Let $\iota:N\hookrightarrow M$ be a CR-submanifold of the Hermitian manifold $(M,\gamma,J)$. Then, the induced structure $F$ of $N$ is a classical CRF structure iff
\begin{equation}\label{eqclind} d\Omega(X,Y,Z)=d\Omega(JX,JY,Z),\;\gamma(b(FX,FY)-b(X,Y),JZ)=0,\end{equation}
for all $X,Y\in TN\cap(JTN)$ and $Z\in TN\cap(J\nu N)$.\end{theorem}
\begin{proof}
As earlier, $\Omega(X,Y)=\gamma(JX,Y)$ is the K\"ahler form. Following \cite{V1}, the structure $F$ is classical CRF iff
\begin{equation}\label{clCRF} [H,H]\subseteq H,\;[H,Q^c]\subseteq H\oplus Q^c,\end{equation}
where $H,\bar{H},Q$ are the $\pm i,0$-eigenbundles of $F$ and the brackets are Lie brackets.

The definition (\ref{clasind}) of $F$ shows that $Q=TN\cap(J\nu N)$, $H\oplus\bar{H}$ is the complexification of $P=im\,F=TN\cap(JTN)$ and $H=L\cap T^cN$, where $L$ is the $i$-eigenbundle of $J$.
In particular, the integrability of $J$ ($M$ is Hermitian) implies the first condition (\ref{clCRF}) and we have to take care of the second condition only.

The second condition (\ref{clCRF}) is equivalent to \cite{V1}
\begin{equation}\label{clCRF2} F[FX,Z]-F^2[X,Z]=0,\;\forall X\in P, Z\in Q. \end{equation}
Because of the second condition (\ref{Fclasmetric}) and since, $F|_P=J_P,F^2|_P=-Id$, (\ref{clCRF2}) is equivalent to
\begin{equation}\label{aux1} \gamma([JX,Z],JY)=\gamma([X,Z],Y),\;\forall X,Y\in P,Z\in Q.
\end{equation}
Indeed, (\ref{aux1}) means that the left hand side of (\ref{clCRF2}) is orthogonal to $P$ and it is also orthogonal to $Q$ because $F|_Q=0$.

We shall express (\ref{aux1}) using the equalities
$$\begin{array}{c}[X,Z]=\nabla^\gamma_XZ-\nabla^\gamma_ZX,\; [JX,Z]=\nabla^\gamma_{JX}Z-\nabla^\gamma_Z(JX),\vspace*{2mm}\\ \nabla^\gamma_Z(JY)=(\nabla^\gamma_ZJ)Y+J\nabla^\gamma_ZY\end{array}$$
and the $\gamma$-compatibility if $J$. The result is
\begin{equation}\label{aux2} \gamma((\nabla^\gamma_ZJ)X,JY)-\gamma(\nabla^\gamma_{JX}Z,JY)+\gamma(\nabla^\gamma_XZ,Y)=0.
\end{equation}

On the other hand, we have
$$\begin{array}{l}-\gamma(\nabla^\gamma_{JX}Z,JY)=\gamma(J\nabla^\gamma_{JX}Z,Y)= \gamma(\nabla^\gamma_{JX}(JZ),Y)-\gamma((\nabla^\gamma_{JX}J)Z,Y),\vspace*{2mm}\\ \gamma(\nabla^\gamma_XZ,Y)= -\gamma(\nabla^\gamma_X(J^2Z),Y)= -\gamma((\nabla^\gamma_XJ)(JZ),Y)+\gamma(\nabla^\gamma_X(JZ),JY),\end{array}$$
and it follows that (\ref{aux2}) is equivalent to
$$\begin{array}{c} \gamma((\nabla^\gamma_ZJ)X,JY) -\gamma((\nabla^\gamma_{JX}J)Z,Y) +\gamma(\nabla^\gamma_{JX}(JZ),Y)\vspace*{2mm}\\ -\gamma((\nabla^\gamma_XJ)(JZ),Y)+\gamma(\nabla^\gamma_X(JZ),JY)=0.\end{array}$$

Then, since $Z\in Q$ implies $JZ\in\nu N$, the third and fifth term of the previous equality may be expressed using the Gauss-Weingarten equations and the relation between the Weingarten operator and the second fundamental form. As a result, we get the following equivalent form of the condition (\ref{clCRF2})
\begin{equation}\label{newclCRF2}\begin{array}{c} \gamma((\nabla^\gamma_ZJ)X,JY) -\gamma((\nabla^\gamma_{JX}J)Z,Y) -\gamma((\nabla^\gamma_XJ)(JZ),Y)\vspace*{2mm}\\ =\gamma(b(JX,Y)+b(X,JY),JZ).\end{array} \end{equation}

To continue, we recall that the integrability of $J$ is equivalent to the following equality \begin{equation}\label{eqdinKN}
2\gamma(\nabla^\gamma_XJ(Y),Z)=d\Omega(X,Y,Z)-d\Omega(X,JY,JZ),\,\forall X,Y,Z\in TM\end{equation}
(this result is given by Proposition IX.4.2 of \cite{KN} with our conventions for the sign of $\Omega$ and the evaluation of the exterior differential). We also recall the equality
\begin{equation}\label{eqptdOmega} d\Omega(JZ,JX,JY)=d\Omega(JZ,X,Y)+d\Omega(Z,JX,Y)+d\Omega(Z,X,JY) \end{equation} (check for arguments of complex type $(1,0),(0,1)$).

Modulo (\ref{eqdinKN}) and (\ref{eqptdOmega}) condition (\ref{newclCRF2}) becomes
\begin{equation}\label{clCRF3} d\Omega(Z,X,JY)+d\Omega(Z,JX,Y)=2[\gamma(b(JX,Y)+b(X,JY),JZ)]. \end{equation}
In (\ref{clCRF3}) the left hand side is skew symmetric in $X,Y$ and the right hand side is symmetric. Therefore, the equality holds iff both of its sides vanish and the replacement of $Y$ by $JY$ shows that the result is exactly (\ref{eqclind}).
\end{proof}
\begin{corol}\label{corol1} If $N$ is a totally umbilical (in particular, totally geodesic) submanifold of a K\"ahler manifold $M$, the induced F structure of $N$ is a classical CRF structure.
\end{corol}
\begin{proof} Under the hypotheses of the corollary, conditions (\ref{eqclind}) obviously hold.
\end{proof}
\begin{corol}\label{corol2} If $N$ is an $\Omega$-coisotropic submanifold of the Hermitian manifold $(M,\gamma,J)$ it is a CR-submanifold and the induced F structure is classical CRF iff the first condition (\ref{eqclind}) holds and the second fundamental form satisfies the equality
$$ b(FX,FY)=b(X,Y),\;\;\forall X,Y\in im\,F.$$
\end{corol}
\begin{proof}
It is well known that $N$ is a CR-submanifold and, in this case, (\ref{restrcond}) takes the form
$$ TN=(TN\cap J(TN))\oplus J(\nu N).$$
Indeed, we have $J(\nu N)=T^{\perp_\Omega}N\subseteq TN$ and the second term of the right hand side of (\ref{restrcond}) is $J(\nu N)$. On the other hand, it follows easily that $(J\nu N)^{\perp_{\gamma'}}=TN\cap J(TN)$. Then, the assertion of the corollary follows from the fact that, in the second condition (\ref{eqclind}), $JZ$ runs through the whole  normal bundle $\nu N$.
\end{proof}

With the preparations done, we now give the answer to the motivating question of the note. It turns out to be a straightforward extension of Proposition 3.3 of \cite{Vnou}.
\begin{theorem}\label{inducingCRFK} Let $\iota:N\hookrightarrow M$ be a generalized CR-submanifold of the generalized K\"ahler manifold $(M,\gamma,\psi,J_\pm)$. Then, the induced generalized metric F structure $\mathcal{F}$ of $N$ is a generalized CRFK structure iff, $\forall Z\in Q_\pm=ker\,F_\pm$, one has
\begin{equation}\label{condindCRFK}\begin{array}{l} d\psi(X,Y,J_\pm Z)=d\psi(J_\pm X,J_\pm Y,J_\pm Z),\;\;\forall X,Y\in P_\pm=im\,F_\pm,\vspace*{2mm}\\ d\psi(X,J_\pm Y,J_\pm Z)=\mp2\gamma(b(X,F_\pm Y),J_\pm Z),\;\;\forall X\in TN,Y\in P_\pm.\end{array} \end{equation} \end{theorem}
\begin{proof}
Since $M$ is generalized K\"ahler, $(\gamma,J_\pm)$ are Hermitian structures, hence, (\ref{eqdinKN}) holds for these two structures, and it implies that condition (\ref{relpsiJ}) is equivalent to \cite{Gualt}
\begin{equation}\label{condGualt} d\psi(X,Y,Z)=\mp d\Omega_\pm(J_\pm X,J_\pm Y,J_\pm Z), \; \forall X,Y,Z\in TM.\end{equation}

For $\mathcal{F}$ to be CRFK, the first required condition, condition (a), is that $(\gamma',F_\pm)$ are classical metric CRF structures, i.e., that conditions (\ref{eqclind}) hold for both structures. Modulo (\ref{condGualt}) the first condition (\ref{eqclind}) becomes the first condition (\ref{condindCRFK}) and the second condition (\ref{eqclind}) is
\begin{equation}\label{aux0} \gamma(b(F_\pm X,F_\pm Y)-b(X,Y),J_\pm Z)=0,\;\;\forall X,Y\in P_\pm,Z\in Q_\pm. \end{equation}

The second required condition, condition (b), is (\ref{CRFK6}) on $(N,\gamma')$, where we may assume $Z\in P_\pm$ since the condition holds trivially if $F_\pm Z=0$. Then, $F_\pm^2Z=-Z$ and, using (\ref{Fclasmetric}), (\ref{CRFK6}) becomes
\begin{equation}\label{CRFK7} \gamma'((\nabla^{\gamma'}_XF_\pm)Y,F_\pm Z)= \mp\frac{1}{2}[d\psi(X,F_\pm Y,F_\pm Z)-d\psi(X,Y,Z)],\end{equation}
with $X,Y\in TN,Z\in P_\pm$.

We consider the cases (i) $Y\in P_\pm$, (ii) $Y\in Q_\pm$ separately.
In case (i), since $F_\pm|_{P_\pm}=J_\pm|_{P_\pm}$, the Gauss equation yields
$$\gamma'((\nabla^{\gamma'}_XF_\pm)Y,F_\pm Z)  =\gamma((\nabla^\gamma_XJ_\pm)Y,J_\pm Z), $$
which makes (\ref{CRFK7}) take the form (\ref{relpsiJ}) with $Z$ replaced by $J_\pm Z$. Therefore it holds because $M$ is generalized K\"ahler.

In case (ii), we have $F_\pm Y=0$ and we get
$$\gamma'(F_\pm(\nabla^{\gamma'}_XF_\pm)Y,Z)=
-\gamma'(F^2_\pm\nabla^{\gamma'}_XY,Z)= -\gamma'(\nabla^{\gamma'}_XY,F^2_\pm Z)=\gamma'(\nabla^{\gamma'}_XY,Z),$$
which, together with the Gauss equation, changes (\ref{CRFK7}) into
\begin{equation}\label{aux01}\gamma(\nabla^{\gamma}_XY,Z) =\pm\frac{1}{2}d\psi(X,Y,Z).\end{equation}

Furthermore, we shall take into account that $Y\in Q_\pm$ implies $J_\pm Y\in\nu N$ and use the Weingarten equation. We get
$$\begin{array}{r}\gamma(\nabla^{\gamma}_XY,Z)=\gamma(J_\pm\nabla^{\gamma}_XY,J_\pm Z)= \gamma(\nabla^{\gamma}_X(J_\pm Y),J_\pm Z)- \gamma((\nabla^{\gamma}_XJ_\pm)Y,J_\pm Z) \vspace*{2mm}\\ \stackrel{(\ref{G-W}),(\ref{relpsiJ})}{=}-\gamma(b(X,J_\pm Z),J_\pm Y) \mp\frac{1}{2}[d\psi(X,J_\pm Y,J_\pm Z)-d\psi(X,Y,Z)].\end{array}$$
Accordingly, condition (\ref{aux01}) becomes
$$\gamma(b(X,J_\pm Z),J_\pm Y)=\pm\frac{1}{2}d\psi(X,J_\pm Y,J_\pm Z),$$
which is the second condition (\ref{condindCRFK}) with the replacements
$Y\mapsto Z,Z\mapsto Y$.

To end the proof we only have to remark that conditions (\ref{condindCRFK}) imply (\ref{aux0}). This follows by noticing that, if we replace $X\in P_\pm$ by $F_\pm X,\,X\in P_\pm$, the second condition (\ref{condindCRFK}) becomes
$$\gamma(b(X,Y),J_\pm Z)=\mp\frac{1}{2}d\psi(X,Y,J_\pm Z),\;,\;\forall X,Y\in P_\pm,Z\in Q_\pm$$
and by using the first condition (\ref{condindCRFK}).
\end{proof}
\begin{corol}\label{corol3} If $M$ is a generalized K\"ahler manifold with a closed associated form $\psi$, then, any totally geodesic, generalized CR-submanifold of $M$ has an induced CRFK structure.\end{corol}
\begin{proof} The assertion is an obvious consequence of conditions (\ref{condindCRFK}). \end{proof}
\begin{corol}\label{corol4} If $M$ is a generalized K\"ahler manifold with a closed associated form $\psi$ and $N$ is a bi-coisotropic submanifold, then, the induced generalized F structure is CRFK iff $b(X,Y)=0$, $\forall X\in TN,Y\in P_\pm$.\end{corol}
\begin{proof} By bi-coisotropic we understand that $N$ is coisotropic with respect to the two K\"ahler forms $\Omega_\pm$. The assertion follows because $J_\pm Z$ of the right hand side of (\ref{condindCRFK}) runs through the whole bundle $\nu N$ (see the proof of Corollary \ref{corol2}).
\end{proof}

Because of the symmetry of the second fundamental form, the CRFK condition of Corollary \ref{corol4} may also be seen as $b(X,Y)=0$, $\forall X\in P_\pm,Y\in TN$. Thus, it follows that, if the induced structure $F$ of the corollary is CRFK, and if $b(Z,Z')=0$ for either $Z,Z'\in Q_+$ or $Z,Z'\in Q_-$, then, $N$ is a totally geodesic submanifold of $M$.

{\small Department of Mathematics, University of Haifa, Israel.\\ E-mail: vaisman@math.haifa.ac.il}
\end{document}